\documentclass{amsart}
\usepackage{amssymb, euscript}

\def\Cal{\mathcal}

\def\P{{\Cal P}}

\def\S{{\Cal S}}

\def\bbr{{\mathbb R}}

\def\bbh{{\mathbb H}}

\def\rank{{\hbox{\rm rank}}}

\def\diag{{\hbox{\rm diag}}}

\def\cos{{\hbox{\rm cos}}}
\def\det{{\hbox{\rm det}}}

\def\min{{\hbox{\rm min}}}

\def\part{\partial}
\def\intl{\int\limits}
\def\b{\beta}

\def\Gam{\Gamma}

\def\a{\alpha}

\def\om{\omega}

\def\del{\delta}
\def\vp{\varphi}

\def\g{\gamma}
\def\gam{\gamma}
\def\Lam{\Lambda}
\def\sig{\sigma}
\def\lam{\lambda}

\font\frak=eufm10

\def\fr#1{\hbox{\frak #1}}
        \def\fra{\fr{a}}

        \def\frh{\fr{h}}

        \def\frk{\fr{k}}
        
\def\frM{\fr{M}}

        \def\frp{\fr{p}}
        \def\frq{\fr{q}}

        \def\fru{\fr{u}}

\def\Ma{\frM_{n,m}}

\def\S{\EuScript{S}}

\def\Ma{M_{n, m}}

\def\f0{f_0}
\def\Fc0{\varphi_0}
\def\rn{\bbr^n}

\def\vnm{V_{n,m}}

\def\p{\P_k}

\def\cos{{\hbox{\rm cos}}}
\def\det{{\hbox{\rm det}}}

\def\p{\P_m}
\def\gm{\Gamma_m}

\numberwithin{equation}{section}

\newcommand{\be}{\begin{equation}}
\newcommand{\ee}{\end{equation}}

\newcommand{\bea}{\begin{eqnarray}}
\newcommand{\eea}{\end{eqnarray}}
\newcommand{\Bea}{\begin{eqnarray*}}
\newcommand{\Eea}{\end{eqnarray*}}

\def\kl{K_\ell}
\def\rl{\bbr^\ell}
\def\rn{\bbr^n}
\def\rnl{\bbr^{n-\ell}}
\def\lv{{\boldsymbol \lam}}

\newtheorem{theorem}{Theorem}
\newtheorem{lemma}[theorem]{Lemma}

\newcommand{\SOn}{\mathrm{SO}(n)}
\newcommand{\GL}{\mathrm{GL}}
\newcommand{\SO}{\mathrm{SO}}

\newcommand{\SL}{\mathrm{SL}}

\def\sideremark#1{\ifvmode\leavevmode\fi\vadjust{\vbox to0pt{\vss
 \hbox to 0pt{\hskip\hsize\hskip1em
\vbox{\hsize2cm\tiny\raggedright\pretolerance10000
 \noindent #1\hfill}\hss}\vbox to8pt{\vfil}\vss}}}%

\theoremstyle{remark}

\numberwithin{equation}{section}

\begin{document}

\title[Invariant Functions]
{Invariant Functions on Grassmannians}

\author{Gestur \'Olafsson}
\address{Department of Mathematics, Louisiana State University, Baton Rouge,
LA, 70803 USA} \email{olafsson@math.lsu.edu}

\author{Boris Rubin}
\address{
Department of Mathematics, Louisiana State University, Baton Rouge,
LA, 70803 USA}

\email{borisr@math.lsu.edu}

\thanks{The first author was supported by the NSF grant
DMS 0402068. The second author was supported in part by the  NSF grant DMS-0556157
and the Louisiana EPSCoR program, sponsored  by NSF and the Board of
Regents Support Fund.}



\dedicatory{This article is dedicated to Professor Sigurdur Helgason
on the occasion of his 80th birthday}


\begin{abstract}
It is known, that every  function  on the unit sphere in $\bbr^n$,
which is invariant under rotations about some coordinate axis, is
completely determined by a function of one variable. Similar
results, when invariance of a function reduces dimension of its
actual argument, hold for every compact symmetric space and can be
obtained in the framework of Lie-theoretic consideration. In the
present
 article, this phenomenon is given precise meaning for functions  on the
Grassmann manifold $G_{n,i}$ of $i$-dimensional
 subspaces of $\bbr^n$, which are invariant under orthogonal transformations
  preserving complementary coordinate
subspaces of arbitrary fixed dimension. The corresponding integral
formulas are obtained. Our method relies on bi-Stiefel decomposition
and does not invoke Lie theory.
\end{abstract}

\maketitle

\section*{Introduction}

\setcounter{equation}{0}

\noindent Integral formulas for semisimple Lie groups and related
symmetric spaces constitute a core of geometric analysis. Many such
formulas are presented in  remarkable books by S. Helgason \cite
{H94}-\cite{H01a}. They are intimately connected with decompositions
of the corresponding Lie algebras and Haar measures, and amount to
pioneering works by H. Weyl, \'E. Cartan, Harish-Chandra; see
bibliographical notes in \cite[p. 231]{H00}. One of such important
formulas is related to the Cartan decomposition $G=KAK$. Its
generalization $G=KAH$ is due to Flensted-Jensen \cite{FJ2} in
 the noncompact case and Hoogenboom \cite{Hoo, Hoo2} for
$G$ compact.

To be more specific, let $U$ be a connected compact real semisimple
Lie group. Let $\theta$ and $\sig$ be two commuting involutions of
$U$, $$U^\theta =\{u\in U \mid \theta(u)=u\}\qquad \text{\rm
(similarly for $U^\sig$).}$$ Let $K$ and $H$ be closed subgroups of
$U$ such that \be\label{sym}(U^\theta)_0\subseteq K\subseteq
U^\theta\qquad\mathrm{and} \qquad (U^\sig)_0\subseteq H\subseteq
U^\sigma,\ee where the subscript ${}_o$ denotes the corresponding
connected component of the unity $e$. Subgroups, which obey
(\ref{sym}), are called symmetric and the quotient spaces $U/K$ and
$U/H$ are compact symmetric spaces. Our interest will be in the
double coset space $K \backslash U/H$.

For the Lie
algebra $\fru$ of $U$, we consider two Cartan decompositions $\fru=\frk +\frp$ and $\fru=\frh +\frq$, where $ \frp$ and $ \frq$
are $-1$ eigenspaces of differentials $d\theta$ and $d\sig$ in
$\fru$, respectively. Let $\fra$ be a maximal abelian subalgebra in
$ \frp \cap \frq$, $A=\exp (\fra)$. Let $M=Z_{K\cap H}(A)$ denote
the centralizer of $A$ in $K\cap H$. According to \cite[ formula
(4.12)]{Hoo}, there is a nonnegative function $\delta $ on $A$, that
can be expressed in terms of $\sin$ and $\cos$ functions, the
restricted roots of $\mathfrak{a}_\mathbb{C}$ in
$\mathfrak{u}_\mathbb{C}$, and the multiplicities, so that
\begin{equation}\label{eq-0.1}
\int_{U/H} f(uH)\, duH=c\int_{A} \int_{K/M} f(kaH) \,\del (a) \,dkM \,da, \qquad f \in C(U/H).
\end{equation}
The constant $c$ can be explicitly evaluated. Here, as elsewhere in
this article, except where clearly stated, the invariant measure on
a compact group is normalized to be one.

Formula (\ref{eq-0.1}) is a consequence of the corresponding
decomposition of the Haar measure $du$ on $U$; see \cite[ Theorem
4.7]{Hoo}. Fundamental results in this directions for  noncompact
Lie groups were first obtained by Berger \cite{Be}; for a modern
account,  see Flensted-Jensen \cite {FJ1, FJ2}, \cite{Ma}.  The
method by Hoogenboom \cite{Hoo}
 gives integral formulas both for the noncompact and compact cases.
 If $f$ is left $K$-invariant then (\ref{eq-0.1}) yields
\be\label{invf} \int_{U/H} f(uH)\, duH=c\int_{A} f_0(a) \,\del (a)
\,da \ee for some function $f_0$ on $A$. The map $f \to f_0$
preserves the smoothness (or integrability) of $f$ up to the weight
function $\del$. Formula (\ref{invf}) can be applied to  the study
of left $K$-invariant functions $f$ on the symmetric space $U/H$.

In the present article, we obtain
 explicit
 characterization of such functions,  when $U$ stands for the orthogonal
 group $\mathrm{O}(n)$,
 $U/H$ is the Grassmann manifold
 $$G_{n,i}=\mathrm{O}(n)/(\mathrm{O}(n-i) \times \mathrm{O}(i))=\SO (n)/\mathrm{S}(\mathrm{O}(n)\times \mathrm{O}(n-i))$$
  of
$i$-dimensional
 subspaces of $\bbr^n$, and the subgroup $K$ has the form
 \bea\label{kael} K\equiv \kl\!&=&\!\left \{\gam \in \mathrm{O}(n)
\, \Big |\, \, \gam\! =\! \left[\begin{array} {cc} \a & 0
\\ 0 & \b \end{array} \right], \quad \a  \!\in  \!\mathrm{O}(n-\ell), \; \b  \!\in \!
\mathrm{O}(\ell) \right \}\\&\sim&  \mathrm{O}(n-\ell) \times
\mathrm{O}(\ell).\nonumber \eea In this setting, $i$ and $\ell$ are
arbitrary
 integers, $1\le i, \ell \le n-1$.
Note that $U=O(n)$, $H=\mathrm{O}(n-i) \times \mathrm{O}(i)$,  and
$K=\mathrm{O}(n-\ell) \times \mathrm{O}(\ell)$ are not connected,
but one can show that (\ref{eq-0.1}) and (\ref{invf}) are still
valid. The subgroup $K_\ell$ is symmetric in the sense that
$K_\ell=\mathrm{O}(n)^{\theta_\ell}$, where the involution
$\theta_\ell $  is defined by
$$\theta_\ell (u)=I_{n-\ell,\ell}uI_{n-\ell,\ell}; \qquad \; u \in
O(n), \quad I_{n-\ell,\ell}=\left[\begin{array} {cc} I_{n-\ell} & 0
\\ 0 & -I_\ell \end{array} \right].$$ In fact, one can readily see that
$$\theta_\ell\left(\left[\begin{array} {cc}  A & B \\ C& D\end{array}\right]\right )=
\left[\begin{array} {cc}  A & -B \\ -C& D\end{array}\right].
$$
Since $I_{n-\ell,\ell}$ and $I_{n-i,i}$ commute, the involutions
$\theta_\ell$ and $\theta_i$ commute too, and the results by
Hoogenboom can be applied.

Unlike the Lie theoretic argument sketched above, our consideration
does not invoke the Lie-theoretic techniques and is motivated by
application to problems in convex geometry, dealing with sections of
star-shaped bodies with symmetries; see, e.g., \cite {R}.
 We  also recall, that if
one would like to use the representation of noncompact semisimple
Lie groups, in particular the principal series representations and
intertwining operators, then $G_{n,i}$ can also be written as $\SL
(n,\mathbb{R})/P_i$, where $P_i$ is the parabolic subgroup
$$P_i=\left\{\left[\begin{array} {cc}  A & X\\
0 & B\end{array}\right]\in\SL (n,\bbr)\, \Big|\, A\in \GL (i,\bbr
),\, B\in \GL (n-i,\bbr) \, ,\,\, X\in\frM_{i,n-i}\right\}$$ and
$\frM_{r,s}\simeq \bbr^{rs}$ stands for the space of $r\times s$
real matrices.

Let us explain the idea of the paper by the simple example
$$ S^{n-1}=\SOn /\SO(n-1)=\mathrm{O}(n)/\mathrm{O}(n-1),$$ where $S^{n-1}$ is the unit sphere in
$\bbr^n$ with the area $\sigma_{n-1}=2\pi^{n/2}/\Gamma (n/2)$. We
fix  an integer $\ell$, $1\le \ell\le n-1$, and  write
\be\label{ooo} \rn=\rnl \oplus \rl, \qquad
\rnl=\overset{n-\ell}{\underset{j=1}{\bigoplus}}\,\bbr e_j, \qquad
\rl=\overset{n}{\underset{j=n-\ell +1}{\bigoplus}}\,\bbr e_j,\ee
where $ \, e_1, e_2, \ldots , e_n$ are the  coordinate unit vectors.
 According to (\ref{ooo}), every point $\theta \in S^{n-1}$ can be
represented in bi-spherical coordinates
as
\be\label{bsph} \theta=\left[\begin{array} {c} u\, \sin \om \\
v\, \cos \,\om \end{array} \right], \qquad u\in S^{n-\ell-1}, \quad
v\in S^{\ell-1}, \quad 0\le \om \le \frac{\pi}{2},\ee so that $
d\theta \!=\!\sin^{n-\ell-1} \om \,\cos^{\ell-1}\om \,du dv d\om$,
where $d\theta, du$, and $dv$ denote the relevant (non-normalized)
volume elements; see, e.g., \cite{VK2}. Clearly, $\cos^{2}\om
\!=\!\theta^t\sig\sig^t\theta\!=\!\theta^t{\rm Pr}_{\rl}\theta$,
where ${\rm Pr}_{\rl}$ denotes the orthogonal projection onto $\rl$
and \be\label{eell}
\sig=[e_{n-\ell +1},\ldots , e_n]=\left[\begin{array} {c} 0 \\
I_\ell \end{array} \right].\ee The following statement is an
immediate consequence of (\ref{bsph}).
\begin{theorem}\label{teo1} For $s\in [0,1]$, let $$ d\nu (s)=s^{\ell/2-1}
(1-s)^{(n-\ell)/2-1} \, ds.$$ An integrable function $f$ on
$S^{n-1}$ is $\kl$-invariant if and only if there is a function
$f_0\in L^1([0,1]; d\nu)$ such that $f(\theta)=f_0(s)$, where
$s^{1/2}=(\theta^t{\rm Pr}_{\rl}\theta)^{1/2}$ is the cosine of the
angle between the unit vector $\theta$ and the coordinate subspace
$\rl$. Moreover,  \bea \intl_{S^{n-1}}f(\theta)\,
d\theta&=&c\,\int_0^{\pi/2} f_0(\cos^{2}\om) \, \sin^{n-\ell-1} \om
\,\cos^{\ell-1}\om \, d\om\nonumber
\\\label{gn1}&=&\frac{c}{2}\,\int_0^{1} f_0 (s) \, d\nu (s), \qquad c=\sig_{\ell
-1}\sig_{n-\ell -1}.\eea
\end{theorem}

\section{Main results}
\noindent Let  $G_{n,i}$  be the Grassmann manifold of
$i$-dimensional linear subspaces  $\xi$ of $\bbr^n$, $1\le i\le
n-1$.  It is assumed, that $G_{n,i}$ is endowed with the
$O(n)$-invariant measure $d\xi$ of total mass $1$. For $1\le \ell
\le n-1$ let $m=\min\{i,\ell\}$. We will need the simplex
\be\label{ssim} \Lam_m\!=\!\{\lv\!=\!(\lam_1, \dots , \lam_m)\mid
1\ge\lam_1\ge \dots \ge\lam_m\ge 0\}, \ee and the  Siegel gamma
function \be\label{2.444}
 \gm (\a)=\pi^{m(m-1)/4}\prod\limits_{j=0}^{m-1} \Gam (\a- j/2). \ee
To every subspace $\xi \in G_{n,i}$, we assign a point $\lv=(\lam_1,
\dots , \lam_m)$ in $\Lam_m$, so that $\lam_1, \dots , \lam_m$ are
eigenvalues of the positive semi-definite matrix
\be\label{forg}r=\left\{
 \begin{array} {ll} \Theta^t{\rm Pr}_{\rl}\Theta & \mbox{if $ i\le \ell$
 ,}\\{}\\
\Psi^t{\rm Pr}_{\xi}\Psi& \mbox{if $ i>\ell$}.\\
  \end{array}\right.\ee
 Here $\Theta =(\theta_{i,j})_{n\times i}$ and $\Psi=(\psi_{i,j})_{n\times \ell}$ are arbitrary fixed
 matrices whose columns form an orthonormal basis in $\xi$ and $\rl$, respectively; $\Theta^t$ and $\Psi^t$ are the corresponding
 transposed matrices; ${\rm Pr}_{\xi}$ and ${\rm Pr}_{\rl}$
 denote the relevant orthogonal projections. Clearly, $\lv$ is
 independent of the choice of orthonormal bases in $\xi$ and $\rl$.

\begin{theorem}\label{teo} Assume that $ 1\le i,\ell \le n-1$ are such that $i+\ell \le n$.
Let $m=\min \{i,\ell\}$.
For  $\lv\in \Lam_m$, we set
$$ d\nu (\lv)=\prod\limits_{1\le
 j<k\le m}(\lam_j-\lam_k)\prod\limits_{j=1}^m \lam_j^\a (1-\lam_j)^\b d
\lam_j,
$$
$$ \a=(n-\ell-i-1)/2, \qquad \b=(|\ell -i| -1)/2.$$
An integrable function $f$ on $G_{n,i}$ is $K_\ell$ -invariant if
and only if there is a function $f_0 \in L^1(\Lam_m ; d\nu )$  such
that $f(\xi)=f_0 (\lv)$, where $\lv$ is formed by eigenvalues of
matrix (\ref{forg}). Moreover,
\be\label{sss}\int_{G_{n,i}}f(\xi)\, d\xi=c \int_{\Lam_m}f_0 (\lv)\, d\nu
(\lv),\ee where \be c=c_m \, \left\{
\begin{array}{lcl}\displaystyle{
\Gamma_i (n/2)/\,\Gamma_i (\ell/2)\, \Gamma_i ((n-\ell)/2)} &
\text{if} &i\le \ell, \\
{}\\
\displaystyle{  \Gamma_\ell (n/2)/\,\Gamma_\ell (i/2)\,
\Gamma_\ell((n-i)/2)} & \text{if} &i\ge \ell,
 \end{array}
\right. \ee
$$ c_m=\pi^{(m^2+m)/4}\Big (\prod\limits_{j=1}^m
j\;\Gam(j/2)\Big )^{-1}. $$
\end{theorem}

The geometrical meaning of  $\lam_1, \dots , \lam_m$ in the equality
$$f(\xi)=f_0 (\lv)\equiv f_0 (\lam_1, \dots , \lam_m)$$ is that $
\lam_1=\cos^2 \om_1, \dots , \lam_m=\cos^2 \om_m$, where $\om_1,
\ldots, \om_m$ are canonical angles, which determine
 the relative position of a subspace $\xi \in G_{n,i}$ with
respect to the coordinate subspace $ \rl$; see, e.g., \cite{C,J}.

The proof of Theorem \ref{teo} relies on the bi-Stiefel
decomposition of the Haar  measure on the Stiefel manifold \cite{GR,
Herz}; see Lemma \ref{l2.1}.  A simple proof of it, presented in
Section 2, is an adaptation of the argument of Zhang \cite{Zh} to
the real case.

We conjecture, that our method extends to the hyperbolic case
$\bbh^n$, when $i$-dimensional planes are substituted by
$i$-dimensional totally geodesic submanifolds of $\bbh^n$. We plan
to study this case in the context of related problems of integral
geometry in forthcoming publications.

\section{The Stiefel manifold and more notation}
\noindent
As before, $\frM_{n,m} \simeq
\bbr^{nm}$ denotes the space of real matrices $x=(x_{i,j})$ having $n$
rows and $m$
 columns;  $dx=\prod^{n}_{i=1}\prod^{m}_{j=1}
 dx_{i,j}$ is the volume element on $\Ma$. Given a square matrix
 $a$,
let $|a|:=|\det (a)|$. Let $\S_m \simeq \bbr^{m(m+1)/2}$  be the space of $m \times m$ real
symmetric matrices $s=(s_{i,j})$
 with the volume element $ds=\prod_{i \le j} ds_{i,j}$. We denote by  $\p\subset \S_m$  the cone of
positive definite matrices in $\S_m$. Given $a$ and  $b$ in  $S_m$,
the symbol $\int_a^b f(s)\, ds$ denotes the integral over the
compact set $(a +\p)\cap (b -\p)$ and $\int_a^\infty f(s)\, ds$
means the integral over  $a +\p$. The group $G=\mathrm{GL}(m,\bbr)$
acts transitively on $\p$ by the rule $g\cdot r := g rg^t$, $g \in
G$.  The corresponding $G$-invariant measure on $\p$ is
\be\label{2.1} d_{*} r = |r|^{-d} dr,  \qquad d= (m+1)/2, \ee
  \cite[p. 18]{T}.
For $n\geq m$, let $\vnm= \{v \in \frM_{n,m}\mid  v^tv=I_m \}$
 be the Stiefel manifold  of orthonormal $m$-frames in $\bbr^n$.
 The group $\mathrm{O}(n)$
 acts transitively on $\vnm$  by the rule $\gam : v\to \gam v, \; \gam \in
\mathrm{O}(n)$, in the sense of matrix multiplication. Let
\begin{equation}\label{sigmam}
\sigma_m=\left[\begin{array} {c} 0 \\
I_m\end{array} \right].
\end{equation}
Most of the time we simply write $\sigma$ for $\sigma_m$. The
stabilizer of $\sigma$ is
$$\mathrm{O}(n-m)\simeq \left\{\left[ \begin{array} {cc}A & 0\\
0 & I_m\end{array}\right]\, \Big|\, A\in \mathrm{O}(n-m)\right\}\,
.$$ Hence $V_{n,m}=\mathrm{O}(n)/\mathrm{O}(n-m)$. We  fix the
corresponding $\mathrm{O}(n)$-invariant measure $dv$ on $\vnm$ so
that
   \be\label{2.16}
 \int_{\vnm} dv = \sigma_{n,m}=\frac {2^m \pi^{nm/2}} {\gm
 (n/2)}, \ee
 $\gm  (\cdot)$ being the  Siegel gamma
function (\ref{2.444});
  see \cite[p. 70]{Mu},   \cite[p.
 57]{J}, \cite[ p. 351]{FK}. The measure $dv$ is also right $\mathrm{O}(m)$-invariant.

Let
$$\frM_{n,m}^*=\{x\in \frM_{n,m}\mid \rank (x)=m\}\, .$$
Then $\frM_{n,m}^*$ is open, dense and of full measure in
$\frM_{n,m}$. Define
\begin{equation}\label{pdo}
\varphi : V_{n,m}\times \p\to \frM_{n,m}^*\, ,\quad (v,r)\mapsto
x=vr^{1/2}\, .
\end{equation}
Then $\varphi $ is surjective and $r=x^t x\in\p$ depends smoothly on
$x$.  The following lemma describes the measure $dx$ on $\frM_{n,m}$
in terms of $V_{n,m}$ and $\p$.

\begin{lemma}\label{l2.3} Assume that $n\ge m$. Let
the notation be as above. Then
$$dx=2^{-m} |r|^{n/2} d_*r \,dv\, .$$
\end{lemma}

The polar decomposition (\ref{pdo}) can be found in many sources,
e.g., in \cite[ p. 482]{Herz}, \cite[p. 93]{GK1}, \cite[pp. 66,
591]{Mu}, \cite[p. 130]{FT}.

The next statement, which is actually due to Zhang \cite{Zh},
contains a higher-rank generalization of the polar decomposition of
the Lebesgue measure in the quarter-plane.
 \begin{lemma}\label{l2.3z}Let $F$ be a function on $\p \times  \p, \;
 d=(m+1)/2$. Then
\bea\label{eit} &&\int_{\p \times  \p} F(p_1, p_2)\,d_*p_1\,
d_*p_2\\&=&\int_0^{I_m}|I_m -r|^{-d}\,d_*r\int_{\p}  F(s^{1/2}r
s^{1/2}, s^{1/2}(I_m -r)s^{1/2})\,d_*s \nonumber\eea provided that
either of these integrals exists in the Lebesgue sense.
\end{lemma}
\begin{proof}
 \bea {\rm l.h.s.}&=&\int_{\p} d_*p_1  \int_{\p} F(p_1, p_1+p_2-p_1)|p_2|^{-d}\,dp_2
 \qquad \text{(set $p_1+p_2=s$)}\nonumber\\ &=&\int_{\p} d_*p_1
 \int_{p_1}^\infty F(p_1, s-p_1)|s-p_1|^{-d}\, ds\nonumber\\ &=&
 \int_{\p}ds\int_0^s  F(p_1, s-p_1)|s-p_1|^{-d}\,d_*p_1\qquad \text{(set $p_1=s^{1/2}r s^{1/2}$)}\nonumber\\ &=&
\int_{\p}d_*s\int_0^{I_m} F(s^{1/2}r s^{1/2}, s^{1/2}(I_m
-r)s^{1/2})\,|I_m -r|^{-d}d_*r,\nonumber\eea and (\ref{eit})
follows.
\end{proof}

\begin{lemma}\label{l2.1} {\rm (bi-Stiefel decomposition)} Let $k$,
$m$, and $n$  be positive integers satisfying $$1 \le k  \le n-1,
\qquad 1 \le m  \le n-1, \qquad m\le \min (k, n-k).$$
 Almost all matrices $v \in \vnm$ can be
represented in the form
 \be\label{herz1}
  v= \left[\begin{array} {cc} u_1 r^{1/2}
\\u_2 (I_m -r)^{1/2}\end{array}
 \right], \qquad  u_1\in V_{n-k,m}, \quad u_2 \in V_{k,m},
 \ee
  so that
   \be\label{2.11}
 \intl_{\vnm} f(v)\, dv
 \!=\!\intl_0^{I_m} d\mu(r) \!\intl_{ V_{n-k, m}}\! du_1 \intl_{ V_{k,
 m}}\!f \left
(\left[\begin{array} {cc} u_1 r^{1/2}
\\u_2 (I_m \!-\!r)^{1/2}\end{array}
 \right]\right )\, du_2, \ee
 \be\nonumber
 d\mu (r)=2^{-m}|I_m
 -r|^{(k-m-1)/2}|r|^{(n-m-k-1)/2}\, dr.
 \ee
\end{lemma}

 \begin{proof}For m=1, this  is a well-known  bi-spherical decomposition \cite[pp. 12,
 22]{VK2}.
For  $k=m$, see  \cite[p. 495]{Herz}, where the result was obtained
using  the Fourier transform technique and Bessel functions of matrix
argument. In \cite{GR},
Herz's proof was extended  to the form presented above and  it was conjectured that there is an alternative  simple proof that
does not need the Fourier transform. Such
a proof was given by  Zhang \cite{Zh}.
For convenience of the reader, we present it here in a slightly different notation.

The result will follow from the  bi-polar decomposition of the
Lebesgue measure on $\frM_{n,m}$. We split  $x \!\in\! \frM_{n,m}$ in
two blocks, so that $x\!=\!\left[\begin{array} {cc} x_1 \\ x_2
\end{array} \right]$, $x_1 \!\in \!\frM_{n-k,m}$, $x_2 \in
\frM_{k,m}$, and write each block in polar coordinates according to
Lemma \ref{l2.3}. This gives \bea \int_{\frM_{n,m}}
f(x)dx&=&\int_{V_{n-k,m}}\,du_1\int_{V_{k,m}}\,du_2\nonumber
\\&\times&\int_{\p \times \p} f \left (\left[\begin{array} {cc} u_1
p_1^{1/2}
\\u_2 p_2^{1/2}\end{array}
\right]\right ) h(p_1, p_2) \,d_*p_1\, d_*p_2,\eea
\be h(p_1, p_2) =2^{-2m} |p_1|^{(n-k)/2}|p_2|^{k/2}.\ee
By (\ref{eit}), the integral over $\p \times
 \p$  transforms as
\be\label{eita} \int_0^{I_m}\,dr\int_{\p}f \left
(\left[\begin{array} {cc} u_1 (s^{1/2}r s^{1/2})^{1/2}
\\u_2 (s^{1/2}(I_m -r)s^{1/2})^{1/2}\end{array}
\right]\right ) \tilde h(r,s) \,\,d_*s,\ee
where $$
\tilde h(r,s)=2^{-2m}|I_m
 -r|^{(k-m-1)/2}|r|^{(n-m-k-1)/2}|s|^{n/2}.
$$
Furthermore, one can write \[ (s^{1/2}r s^{1/2})^{1/2}=\gam_1
r^{1/2} s^{1/2}, \quad (s^{1/2}(I_m -r)s^{1/2})^{1/2}=\gam_2 (I_m
-r)^{1/2} s^{1/2},\] for some $\gam_1, \gam_2 \in \mathrm{O}(m)$ (just note
that $|\gam_1|=|\gam_2|=1$). Hence, changing the order of
integration, and using right $\mathrm{O} (m)$-invariance of $du_1$ and $du_2$,
we easily get
\bea \int_{\frM_{n,m}}
f(x)\, dx&=&2^{-m}\int_{\p}|s|^{n/2}\, d_*s\int_0^{I_m}\,  d\mu (r)
\nonumber
\\&\times& \label{niz}\int_{V_{n-k,m}}\,du_1\int_{V_{k,m}} f \left
(\left[\begin{array} {cc} u_1 r^{1/2}
\\u_2 (I_m -r)^{1/2}\end{array}
 \right]\,s^{1/2}\right ) \, du_2,\eea
where $d\mu (r)=2^{-m}|I_m
 -r|^{(k-m-1)/2}|r|^{(n-m-k-1)/2}\, dr$. On the other hand, by Lemma
 \ref{l2.3},
\be\label{nez}\int_{\frM_{n,m}} f(x)\, dx=2^{-m}\int_{\p}|s|^{n/2}\, d_*s
\int_{V_{n,m}} f (vs^{1/2}) \,dv.\ee Comparing (\ref{niz}) and
(\ref{nez}), we get the result.

\end{proof}

\section{Proof of Theorem \ref{teo}}

\noindent
We divide the proof into two parts.
\smallskip

\noindent \textbf{Part I.} Let $i\le \ell$, let $\sigma_i$ be as in
(\ref{sigmam}), and let
$$\kappa_i :  V_{n,i}=\mathrm{O}(n)/\mathrm{O}(n-i)\to
G_{n,i}=\mathrm{O}(n)/\mathrm{O}(n-i)\times \mathrm{O}(i) \, ,\quad
g\cdot \sigma_i \to g\mathbb{R}^{i}$$ be the canonical map.
 It is obviously
$\mathrm{O}(n)$-equivariant. We use this to identify a  $K_\ell$
-invariant function  $f$ on $G_{n,i}$ with the left
$K_\ell$-invariant and right $\mathrm{O}(i)$-invariant function
$\vp=f\circ \kappa_i$ on the Stiefel manifold $V_{n,i}$. Then
\be\label{sot} \int_{G_{n,i}} f(\xi)\,
d\xi=\frac{1}{\sig_{n,i}}\int_{V_{n,i}}\vp (\Theta)\, d\Theta.\ee By
Lemma \ref{l2.1} (with $m=i, \; k=\ell$), almost all $\Theta \in
V_{n,i}$ can be represented
as \be\label{sin} \Theta=\left[\begin{array} {cc} u_1 r^{1/2} \\
u_2 (I_i -r)^{1/2}
 \end{array} \right], \qquad u_1 \in V_{n-\ell, i}, \quad u_2\in V_{\ell,
 i}, \quad r \in (0, I_i). \ee
Since $\vp$ is left $K_\ell$ -invariant, it follows that $\vp$ is independent
of the bi-Stiefel coordinates $u$ and $v$ in (\ref{sin}). Thus we we can
write $\vp (\Theta)\equiv \vp_0 (r)$. By (\ref{2.11})
and (\ref{2.16}) we get
\be\label{spo}
\frac{1}{\sig_{n,i}}\int_{V_{n,i}}\vp (\Theta)\, d\Theta=c\,
\int_0^{I_i} \vp_0 (r)\, d\nu_1(r),\ee where \be d\nu_1(r)=|I_i
-r|^{(\ell -i-1)/2}|r|^{(n-\ell)/2}  \,d_*r,\ee \be
c=\frac{\sig_{n-\ell,i}\,\sig_{\ell,i}}{2^{i}\,
\sig_{n,i}}=\frac{\Gamma_i (n/2)}{\Gamma_i (\ell/2)\, \Gamma_i
((n-\ell)/2)}.\ee
Since $r$ can be expressed through $\Theta$ as
$
r=\Theta^t\sig \sig^t \Theta$
and $\vp$  is right
 $\mathrm{O}(i)$-invariant we obtain
\[
\vp_0 (r)= \vp_0 (\Theta^t\sig \sig^t \Theta)=\vp (\Theta)=\vp (\Theta
g)=\vp_0 (\g^t\Theta^t\sig \sig^t \Theta \g)= \vp_0 (\g^tr\g)\] for any
$\g\in
 \mathrm{O}(i)$. Hence, if we write $r$ in polar coordinates
\begin{equation}\label{b3} r=\g ^t\lv\g, \quad \g\in \mathrm{O}(i), \quad
\lv=\diag (\lam_1, \dots ,
\lam_i),
 \end{equation}
where $\lam_1, \dots , \lam_i$ are eigenvalues of $r$, we  get $
\vp_0 (r)=\vp_0 (\lv)$. Moreover, using the known formula for the
invariant measure
\be d_*r=c_i \,\prod\limits_{1\le
 j<k\le i}(\lam_j
-\lam_k)\Big(\prod\limits_{j=1}^i \lam_j^{-(i+1)/2}\,d\lam_j \Big)
d\g \ee
where
$$ c_i=\pi^{(i^2+i)/4}\Big (\prod\limits_{j=1}^i
j\;\Gam(j/2)\Big )^{-1},
$$(see \cite[ p. 23, 43]{T}), we obtain \bea \int_0^{I_i} \vp_0
(r)\, d\nu_1(r)&=& c_i \int_{\Lam_i}
\vp_0 (\lv)\prod_{1\le j<k\le i}(\lam_j -\lam_k)\nonumber \\
&\times& \prod_{j=1}^i \lam_j^{(n-\ell -i -1)/2}\, (1-\lam_j)^{(\ell
-i -1)/2}\, d\lam_j.\nonumber \eea
 Combining this formula with
(\ref{sot}) and (\ref{spo}), we obtain (\ref{sss}).

Conversely, to every function $f_0$ on $\Lam_i$, we can assign a
function $ \vp_0$ on $(0, I_i)$ by the rule $\vp_0 (r)= f_0 (\lam_1,
\dots , \lam_i)$, where $\lam_1, \dots , \lam_i$ are eigenvalues of
$r$ arranged as $1\ge\lam_1\ge \dots \ge\lam_i\ge 0$. The function $
\vp_0$ is
 $\mathrm{O}(i)$-invariant, i.e.,
 \be\label{fora} \vp_0 (r)=\vp_0 (g^trg)\quad \text{\rm for any
$g\in \mathrm{O}(i)$},\ee because  $r$ and $g^trg$ have the same
eigenvalues. Next we define a function $\vp$ on $V_{n,i}$ by $\vp
(\Theta)=\vp_0 (\Theta^t\sig \sig^t \Theta)$, where
$\sig=\sigma_\ell \in V_{n,\ell}$. The function  $\vp$ is left $K_\ell$
-invariant, because for any $\gam \in K_\ell$ we have $\gam^t\sig
\sig^t \gam=\sig \sig^t$ and, therefore, $$ \vp (\gam\Theta)=\vp_0
(\Theta^t\gam^t\sig \sig^t \gam\Theta)=\vp_0 (\Theta^t\sig
\sig^t\Theta)=\vp (\Theta).$$ It remains to note that, owing to
(\ref{fora}), $\vp$ is right
 $\mathrm{O}(i)$-invariant and  can therefore be identified with a $K_\ell$-invariant function on $G_{n,i}$.
\smallskip

\noindent
\textbf{Part II.} Let $i\ge \ell$. Suppose that $f$ is a $K_\ell$
-invariant function on $G_{n,i}$, which is identified with a
function $\vp (\Theta)$ on   $V_{n,i}$ as in Part I. Let
$\Theta=\left[\begin{array} {ll} a
\\ b \end{array}
 \right]$, $a \in \frM_{n-\ell, i}$, $b=\sig^t\Theta\in \frM_{\ell,
 i}$. Since $n-\ell \ge i$, by Lemma \ref {l2.3},  we can write $a$ in polar
 coordinates as $a=vr^{1/2}$, $v \in V_{n-\ell,i}$, $r=a^ta=I_i - b^tb$.
 Fix any $v_0$ in $V_{n-\ell,i}$ and set $v=\a v_0$, $\a\in \mathrm{O}(n-\ell)$.
 Then $$
\Theta=\left[\begin{array} {cc}\a v_0 r^{1/2}
\\ b \end{array}
 \right]= \left[\begin{array}{ll}  \a&0\\
0&I_\ell
\end{array}\right]
 \left[\begin{array} {cc} v_0 (I_i -b^tb)^{1/2}
\\ b \end{array}
 \right].$$
Since $\vp$ is left $K_\ell$ -invariant, then \be\label{bir}
\vp(\Theta)=\vp\left (\left[\begin{array} {cc} v_0 (I_i -b^tb)^{1/2}
\\ b \end{array}
 \right]\right )\equiv \vp_1 (b), \qquad b=\sig^t\Theta.\ee We write
 the transpose of $b$ in polar coordinates $$b^t=us^{1/2}, \qquad
 u\in V_{i,\ell}, \qquad s=bb^t\in\P_\ell.$$
Then we fix any $u_0\in V_{i,\ell}$ and replace $u$ by $\b u_0$ for
some $\b\in \mathrm{O}(i)$. This gives $b=s^{1/2}u^t_0\b^t$ and
therefore, since $\vp$  is right
 $\mathrm{O}(i)$-invariant, $$
\vp (\Theta)=\vp (\Theta\b)=\vp_1 (\sig^t\Theta\b)=\vp_1 (b\b)=\vp_1
(s^{1/2}u^t_0\b^t\b)=\vp_1 (s^{1/2}u^t_0)\, .$$
It means that $\vp
(\Theta)$ is actually a function of $s$. Denote it by $\vp_0
(s)$. Since $b=\sig^t\Theta$ the positive definite matrix
$s=bb^t=\sig^t\Theta\Theta^t\sig=\sig^t{\rm Pr}_\xi \sig$ lies in
the ``interval" $(0, I_\ell)$. Here ${\rm Pr}_\xi \sig$ denotes the
orthogonal projection of $\sig\in V_{n,\ell}$ onto the subspace
$\xi\in G_{n,i}$. Thus
$$ f(\xi)\equiv \vp(\Theta)=\vp_0
(\sig^t\Theta\Theta^t\sig)=\vp_0 (\sig^t{\rm Pr}_\xi \sig)\, .$$ Since
$\vp$ is left-invariant under left translation by
$\tilde \b=\left[\begin{array}{cc}  I_{n-\ell}&0\\
0&\b
\end{array}\right]$, $\b\in \mathrm{O}(\ell)$, i.e., $f(\tilde \b
\xi)=f(\xi)$, it follows that $\vp_0 (\sig^t{\rm Pr}_{\tilde \b\xi} \sig)=\vp_0
(\sig^t{\rm Pr}_{\xi} \sig)$, and therefore
\be\label{the} \vp_0
(s)=\vp_0 (\sig^t{\rm Pr}_{\xi} \sig)=\vp_0 (\sig^t{\rm Pr}_{\tilde
\b\xi} \sig)=\vp_0 (\sig^t\tilde \b\Theta\Theta^t\tilde \b^t
\sig).\ee
Since
$$
\tilde \b^t \sig=\left[\begin{array}{cc}  I_{n-\ell}&0\\
0&\b^t
\end{array}\right]\left[\begin{array} {ll} 0
\\ I_\ell \end{array}
 \right]=\left[\begin{array} {ll} 0
\\ \b^t \end{array}
 \right]=\left[\begin{array} {ll} 0
\\ I_\ell \end{array}
 \right]\b^t=\sig\b^t,$$
equation (\ref{the}) implies that for
all $\beta \in \mathrm{O}(\ell )$ we have
 \be\label{gugu}
 \vp_0
(s)=\vp_0 (\b\sig^t\Theta\Theta^t\sig\b^t)=\vp_0 (\b s\b^t)\, .\ee
Thus, $\vp_0$ depends only on the eigenvalues $\lam_1, \dots ,
\lam_\ell$ of $s=\sig^t{\rm Pr}_{\xi} \sig$, $\vp_0 (s)=\vp_0
(\diag(\lam_1, \dots , \lam_\ell))$. Finally, we write
$$f(\xi)=\vp_0
(\diag(\lam_1, \dots , \lam_\ell))\equiv f_0(\lam_1, \dots , \lam_\ell).$$

For the corresponding integral formula we have
\bea
I&=&\int_{G_{n,i}}f(\xi)\,
d\xi=\frac{1}{\sig_{n,i}}\int_{V_{n,i}}\vp (\Theta)\, d\Theta
\nonumber\\
&=&\frac{1}{\sig_{n,i}}\int_{V_{n,i}}\vp_0
(\sig^t\Theta\Theta^t\sig)\, d\Theta
\nonumber\\
&&\text{\rm(replace $\Theta$ by $g^t\sig_i, \; g\in \mathrm{O}(n),
\quad \sig_i=\left[\begin{array} {ll} 0
\\ I_i \end{array}
 \right]\in V_{n,i})$}\nonumber\\
&=&\int_{\mathrm{O}(n)} \vp_0 (\sig^tg^t\sig_i\sig^t_i g \sig)\,
dg=\frac{1}{\sig_{n,\ell}}\int_{V_{n,\ell}}\vp_0
(\Psi^t\sig_i\sig^t_i \Psi)\, d\Psi.\nonumber\eea
The last integral can be written in bi-Stiefel coordinates by setting
$$\Psi=\left[\begin{array} {cc} u_1  r^{1/2}
\\ u_2 (I_\ell -r)^{1/2}\end{array}
 \right], \qquad u_1\in V_{n-i,\ell}, \quad u_2 \in V_{i,\ell}, \quad r=\Psi^t\sig_i\sig^t_i \Psi. $$ Then Lemma \ref{l2.1} gives
\be\label{kar} I=c\,\int_0^{I_\ell}\vp_0 (r) \, d\nu_2 (r),\ee where
\be d\nu_2(r)=|I_\ell -r|^{(i-\ell -1)/2} |r|^{(n-\ell -i-1)/2}
\,dr,\ee \be c=\frac{\sig_{n-i,\ell}\, \sig_{i,\ell}}{2^\ell
\,\sig_{n,\ell}}=\frac{\Gamma_\ell (n/2)}{\Gamma_\ell (i/2)\,
\Gamma_\ell ((n-i)/2)}.\ee
Since $\vp_0$ is $\mathrm{O}(\ell)$-invariant (see
(\ref{gugu})), we can write  (\ref{kar}) in polar coordinates
in the required form (\ref{sss}).

Conversely, as in Part I, every function $f_0$ on $\Lam_\ell$ can be
associated with an $\mathrm{O}(\ell)$-invariant function $\vp_0$ on
$(0,I_\ell)$, and the latter generates a function $\vp
(\Theta)=\vp_0(\sig^t\Theta\Theta^t\sig)$. This function is left
$K_\ell$ -invariant on $V_{n,i}$. Indeed, let $\gam=\left[\begin{array}{ll}  \a&0\\
0&\b
\end{array}\right]\in K_\ell$, $\a \in \mathrm{O}(n-\ell)$, $\b \in \mathrm{O}(\ell)$.
Then $\vp (\gam\Theta)=\vp_0(\sig^t\gam\Theta\Theta^t\gam^t\sig)$.
Since $$
\sig^t\gam=[0,I_\ell]\left[\begin{array}{ll}  \a&0\\
0&\b
\end{array}\right]=[0,\b]=\b\sig^t,$$ then
 $\vp (\gam\Theta)=\vp_0(\b\sig^t\Theta\Theta^t\sig\b^t)=\vp_0(\sig^t\Theta\Theta^t\sig)=\vp
 (\Theta)$. Furthermore, $\vp
 (\Theta)$ is obviously right $\mathrm{O}(i)$-invariant, and therefore,
 can be identified with a $K_\ell$ -invariant function on $G_{n,i}$.

\end{document}